\documentclass[11pt]{amsart}
\usepackage{amscd}
\usepackage{amsmath}
\usepackage{amsfonts}
\newcommand{\al}{\alpha}
\newcommand{\be}{\beta}
\newcommand{\ga}{\gamma}
\newcommand{\de}{\delta}

\newcommand{\del}{\partial}

\newcommand{\grad}{\mathrm{grad}}
\newcommand{\ov}[1]{\overline{#1}}
\newcommand{\tr}[2]{\textrm{tr}_{#1} #2}

\newcommand{\ti}[1]{\tilde{#1}}
\newcommand{\ve}{\varepsilon}

\begin{document}
\newcounter{remark}
\newcounter{theor}
\setcounter{remark}{0}
\setcounter{theor}{1}
\newtheorem{claim}{Claim}
\newtheorem{theorem}{Theorem}[section]
\newtheorem{prop}{Proposition}[section]
\newtheorem{lemma}{Lemma}[section]
\newtheorem{defn}{Definition}[theor]
\newtheorem{cor}{Corollary}[section]
\newenvironment{remark}[1][Remark]{\addtocounter{remark}{1} \begin{trivlist}
\item[\hskip
\labelsep {\bfseries #1  \thesection.\theremark}]}{\end{trivlist}}
\setlength{\arraycolsep}{2pt}

\title[SCHWARZ LEMMA]{A GENERAL SCHWARZ LEMMA FOR ALMOST-HERMITIAN MANIFOLDS}
\author[VALENTINO TOSATTI]{Valentino Tosatti}

\begin{abstract}  We prove a version of Yau's Schwarz Lemma for general almost-complex manifolds
equipped with almost-Hermitian metrics.
 This requires an extension to this setting of the Laplacian comparison 
theorem. As an application we show that the product of two almost-complex manifolds does not admit
any complete almost-Hermitian metric with bisectional curvature bounded between two negative constants
that satisfies some additional assumptions.
\end{abstract}

\thanks{Part of this work was carried out while the author was visiting the Mathematics
Department of UCLA and the Morningside Center of Mathematics in Beijing; 
the author is supported in part
by a Jean de Valpine Fellowship.}

 \address{Department of Mathematics \\ Harvard University \\ Cambridge, MA 02138}

  \email{tosatti@math.harvard.edu}

\maketitle
\section{Introduction}
The classical Schwarz-Pick lemma says that a holomorphic map from the unit disc in the
complex plane into itself decreases the Poincar\'e metric. This was later extended by
Ahlfors \cite{Ah} to maps from the disc into a hyperbolic Riemann surface, and by
Chern \cite{Che} and Lu \cite{Lu} to allow more general domains and targets.
A major advance was Yau's Schwarz Lemma \cite{Ya2}, which says that 
a holomorphic map from a complete K\"ahler manifold with Ricci curvature bounded below 
into a Hermitian manifold with holomorphic bisectional curvature bounded above by a negative constant,
is distance decreasing up to a constant depending only on these bounds. This proved to be
extremely useful in differential geometry and complex analysis (see for example \cite{LSY}). Later generalizations
of this result were mainly in two directions: relaxing the curvature hypothesis or the K\"ahler assumption
(see \cite{Ro}, \cite{Ch}, \cite{GH2}) or proving similar results for harmonic
maps of Riemannian manifolds \cite{GH1}.

Here we take a different direction and generalize Yau's Schwarz Lemma to the case
when the complex structures are not integrable. Recently there has been a lot
of interest on geometric and analytic aspects of almost-complex manifolds 
(\cite{IR}, \cite{TWY}), also in relation with symplectic geometry (\cite{D}, \cite{W})
and complex analysis \cite{CGS}. Our setting is as follows:
suppose we are given two almost-complex manifolds
$M$ and $\ti{M}$, equipped with Riemannian metrics compatible with the almost complex structures
(we call such data an almost-Hermitian manifold). A map from $M$ to $\ti{M}$ is said to be almost-complex
or holomorphic,
if its differential intertwines the two almost-complex structures. On any almost-Hermitian manifold
there is a preferred choice of connection, the so-called canonical connection,
that generalizes the Chern connection in the integrable case. In general it
is different from the Levi-Civita connection, so it has nontrivial torsion, but 
is more suited for analytic questions \cite{TWY}. From now on, all geometric
quantities (Ricci and bisectional curvature, torsion, etc.) will be the ones of the canonical connection.
With this setup, we have the following
Schwarz Lemma (see Section 2 for notation):

\begin{theorem}\label{main1}
Let $(M^{2n},J,g)$ be a complete almost-Hermitian manifold with second Ricci curvature bounded from below by $-K_1$,
and with torsion and $(2,0)$ part of the curvature bounded.
Let $(\ti{M}^{2\ti{n}},\ti{J},\ti{g})$ be an almost-Hermitian manifold with bisectional curvature bounded from above by $-K_2$, $K_2>0$.
If $f:M\to \ti{M}$ is a non-constant almost-complex map, then we must have $K_1\geq 0$ and
$$f^*\ti{g} \leq \frac{K_1}{K_2} g.$$ In particular if $K_1\leq 0$ then any almost-complex map is constant.
\end{theorem}

\begin{cor}
Let $(M^{2n},J,g)$ be a complete almost-Hermitian manifold with non-negative second Ricci curvature
and with torsion and $(2,0)$ part of the curvature bounded. Then
$M$ doesn't admit any non-constant bounded $J-$holomorphic function $f:M\to\mathbb{C}$.
\end{cor}

Notice that when $M$ is compact the assumptions of bounded torsion and $(2,0)$ part of the curvature
are automatically satisfied. Also, while almost-complex maps between K\"ahler manifolds
are always harmonic \cite{Li}, this is no longer true for general almost-complex manifolds (see (9.11) in \cite{EL}), so
that the results of \cite{GH1} don't apply in our situation.\\

Next we assume that $M$ and $\ti{M}$ have the same dimension. A map $f:M\to \ti{M}$ is called \emph{non-degenerate} 
if $f^*dV_{\ti{g}}$ is a volume form on $M$, and \emph{totally degenerate} if $f^*dV_{\ti{g}}$ vanishes
identically. Then we have the following Schwarz Lemma for the volume forms:

\begin{theorem}\label{main2}
Let $(M^{2n},J,g)$ be a complete almost-Hermitian manifold with second Ricci curvature bounded from
 below, with torsion and $(2,0)$ part of the curvature bounded, and with scalar curvature
bounded from below by $-nK_1$.
Let $(\ti{M}^{2n},\ti{J},\ti{g})$ be an almost-Hermitian manifold of the same dimension $2n$ with first Ricci curvature bounded from above by $-K_2$, $K_2>0$.
If $f:M\to \ti{M}$ is a non-degenerate almost-complex map, then we must have $K_1\geq 0$ and
$$f^* dV_{\ti{g}} \leq \left(\frac{K_1}{K_2}\right)^n dV_g.$$ In particular if $K_1\leq 0$ then any 
almost-complex map is totally degenerate.
\end{theorem}

As an application of the Schwarz Lemma, we study the geometry of the product of two (nontrivial)
almost-complex manifolds. A classical theorem of Preissman implies that a compact Riemannian
manifold with negative sectional curvature cannot be topologically a product manifold. 
For K\"ahler manifolds the notion of bisectional curvature is more natural, and it's easy
to see that a compact K\"ahler manifold with negative bisectional curvature cannot be
the product of two nontrivial complex manifolds (this is because the negativity of
the curvature implies that the cotangent bundle is ample).
When the two factors are allowed to be noncompact, there are similar results due to Yang, Zheng
and Seshadri (\cite{Yn}, \cite{Zh}, \cite{Se}). In \cite{SZ} it is proved that the product of two
complex manifolds doesn't admit any complete Hermitian metric with bounded torsion and bisectional
curvature bounded between two negative constants. It is natural to expect that such a result should
hold in the almost-complex case, and this is precisely what we prove.

\begin{theorem}\label{main3}
Let $M=X\times Y$ be the product of two almost-complex manifolds of positive dimensions. Then $M$
doesn't admit any complete almost-Hermitian metric with torsion and $(2,0)$ part of the curvature 
bounded and
with bisectional curvature bounded between two negative constants.
\end{theorem}

\begin{cor}\label{cor1}
The product of two compact nontrivial almost-complex manifolds doesn't admit any
almost-Hermitian metric with negative bisectional curvature.
\end{cor}

Let us stress that here the bisectional curvature is the one of the canonical
connection, and in general is different from the one of the Levi-Civita connection
(as defined in \cite{Gr} for example). Nevertheless, this curvature is more natural
on almost-Hermitian manifolds (see the discussion after Lemma \ref{comparelap}).\\

The proof of the Schwarz Lemma employs Cartan's formalism of moving frames
and the canonical connection, as in \cite{TWY}. To deal with the case of
noncompact manifolds, we generalize Yau's maximum principle \cite{Ya1}
to our situation. The proof of this requires a suitable Laplacian comparison
theorem for almost-Hermitian manifolds. This is the key technical tool
and is proved along the lines of the classical Laplacian comparison,
but using local holomorphic discs instead of complex coordinates
that are not available, and keeping carefully track of the torsion.
The proof of Theorem \ref{main3} follows the argument in \cite{SZ},
once the Schwarz Lemma and the maximum principle hold.
The paper is organized as 
follows: in section 2, we give some background on almost-Hermitian metrics and the canonical
connection. In section 3, we study the Laplacian of the canonical connection. In section 4, we
prove the Laplacian comparison theorem and the maximum principle. In section 5, we give
a proof of Theorems \ref{main1} and \ref{main2}. Finally, Theorem \ref{main3} is proved in section 6.

\bigskip
\noindent
{\bf Acknowledgements.} \ First of all I would like to thank my thesis advisor Professor S.-T. Yau
for suggesting this problem, sharing his ideas with me and for his constant support. 
I would also like to thank Ben Weinkove for many useful discussions and for reading an earlier version of this paper,
and Harish Seshadri for suggesting this generalization of his results.

\section{Almost-Hermitian manifolds and the canonical connection}
In this section we give some background on almost-Hermitian manifolds, the canonical
connection and its torsion and curvature. Some of the exposition follows \cite{TWY}, section 2.\\

Let $(M,J,g)$ be an almost-Hermitian manifold of dimension $2n$.  Namely, $J$ is an almost complex structure on $M$ and $g$ is a Riemannian metric satisfying
$$g(J X, JY) = g(X,Y),$$
for all tangent vectors $X$ and $Y$. Write $T^{\mathbb{R}}_p M$ 
for the (real) tangent space of $M$ at a point $p$.  In the following we will drop the subscript $p$. 
 Denote the complexified tangent space by $T^{\mathbb{C}}M = T^{\mathbb{R}} M \otimes \mathbb{C}$. 
 Extending $g$ and $J$ linearly to $T^{\mathbb{C}}M$, we see 
that the complexified tangent space can be decomposed as
$$T^{\mathbb{C}} M = T' M \oplus T''M,$$
where $T'M$ and $T''M$ are the eigenspaces of $J$ corresponding to eigenvalues $\sqrt{-1}$ and $- \sqrt{-1}$ respectively.  
$T'M$ and $T''M$ are complex vector spaces of dimension $n$, which inherit a Hermitian metric induced by $g$.  Note that by extending $J$ to forms, we can uniquely decompose $m$-forms into $(p,q)$-forms for each $p$,$q$ with $p+q=m$.
The real tangent bundle $T^\mathbb{R} M$ can be identified with $T'M$ in a natural way,
by sending a vector $X_{\mathbb{R}}$ to $X=\frac{1}{\sqrt{2}}(X_{\mathbb{R}}-\sqrt{-1}JX_{\mathbb{R}})$. This identification
is an isomorphism of complex vector bundles, and an isometry. From now on we'll write $g$ for the induced Hermitian metric on $T'M$,
and $dV_g$ for its corresponding volume element. Choose a local unitary frame $\{ e_1, \ldots, e_n \}$ for $T'M$, and let $\{ \theta^1, \ldots, \theta^n \}$ 
be a dual coframe.  Then we can write
$g =   \theta^i \otimes \ov{\theta^i}$ and $dV_g=(\sqrt{-1})^n\theta^1\wedge\ov{\theta^1}\wedge\dots\wedge\theta^n\wedge\ov{\theta^n}$,
where here, and henceforth, we are summing over repeated indices.

Let $\nabla$ be an affine connection on $T^{\mathbb{R}} M$, which we extend linearly to $T^{\mathbb{C}} M$.  We say that $\nabla$ is an \emph{almost-Hermitian connection} if
$$\nabla J = \nabla g = 0.$$
It is easy to see that such connections always exist on any almost-Hermitian manifold,
and from now on we shall assume that $\nabla$ satisfies this condition.  Observe that 
$J ( \nabla e_i) = \sqrt{-1} \nabla e_i,$ 
and hence $\nabla e_i \in T'M \otimes (T^{\mathbb{C}}(M))^*$.  Then locally there exists a matrix of complex valued 1-forms $\{ \theta^j_i \}$, called the \emph{connection 1-forms}, such that
$$\nabla e_i = \theta^j_i e_j.$$
Applying $\nabla$ to $g(e_i, \ov{e_j})$ and using the condition $\nabla g=0$ we see that $\{ \theta^j_i \}$ satisfies the skew-Hermitian property
$$\theta^j_i + \ov{\theta^i_j} = 0.$$
Now define the \emph{torsion} $\Theta = (\Theta^1, \ldots, \Theta^n)$ of $\nabla$ by
\begin{equation} \label{eqnstructure1}
d \theta^i = - \theta^i_j \wedge \theta^j + \Theta^i, \qquad \textrm{for } i=1, \ldots, n.
\end{equation}
 Notice that the $\Theta^i$ are 2-forms.  Equation (\ref{eqnstructure1}) is known as the \emph{first structure equation}.
Define the \emph{curvature} $\Omega = \{\Omega_j^i \}$   of $\nabla$ by
\begin{equation} \label{eqnstructure2}
d \theta^i_j = - \theta^i_k \wedge \theta^k_j + \Omega^i_j.
\end{equation}
Note that $\{ \Omega^i_j \}$ is a  skew-Hermitian matrix of 2-forms.  Equation (\ref{eqnstructure2}) is known as the \emph{second structure equation}. We have the following lemma (see e.g. \cite{Ga}).

\begin{lemma}
There exists a unique almost-Hermitian connection $\nabla$ on $(M,J,g)$ whose torsion $\Theta$ has everywhere vanishing $(1,1)$ part.
\end{lemma}

We call such a connection the \emph{canonical connection}.
In Riemannian geometry the torsion of a connection $\nabla$ is usually defined by 
\begin{equation}\label{need1}
\nabla_X Y=\nabla_Y X+[X,Y]+\tau(X,Y).
\end{equation}
We'll show in Lemma \ref{comparelap} that in $T^\mathbb{C}M$ the following identity holds:
\begin{equation}\label{need2}
\tau=2(\Theta^i e_i+\ov{\Theta^j e_j}).
\end{equation}
Define functions $M^i_{jk}$ and  $N^i_{\ov{j} \, \ov{k}}$ by
\begin{eqnarray*}
(\Theta^i)^{(2,0)}  =  M^i_{jk} \theta^j \wedge \theta^k, \quad (\Theta^i)^{(0,2)}  =  N^i_{\ov{j} \, \ov{k}} \ov{\theta^j} \wedge \ov{\theta^k},
\end{eqnarray*}
with $M^i_{jk} = - M^i_{kj}$ and $N^i_{\ov{j} \, \ov{k}} = - N^i_{\ov{k} \, \ov{j}}$.
Define $R_{ik \ov{\ell}}^j$,  $K^i_{jk\ell}$ and $K^i_{j \ov{k} \, \ov{\ell}}$ by
\begin{eqnarray*}
(\Omega_i^j)^{(1,1)} & = & R_{ik \ov{\ell}}^j \theta^k \wedge \ov{\theta^\ell} \\
(\Omega^i_j)^{(2,0)} & = & K^i_{jk\ell} \theta^k \wedge \theta^\ell\\
(\Omega^i_j)^{(0,2)} & = & K^i_{j \ov{k} \, \ov{\ell}} \ov{\theta^k} \wedge \ov{\theta^\ell},
\end{eqnarray*}
with $K^i_{jk\ell} = - K^i_{j\ell k}$ and  $K^i_{j\ov{k} \, \ov{\ell}} = - K^i_{j\ov{\ell} \, \ov{k}}$.
The fact that $\{ \Omega^i_j\}$ is skew-Hermitian implies that
\begin{equation}\label{switch}
K^i_{jk\ell}=\ov{K^j_{i\ov{\ell}\,\ov{k}}},\quad R^i_{jk\ov{\ell}}=\ov{R^j_{i\ell\ov{k}}}.
\end{equation}
If $X,Y$ are two $(1,0)$ vectors, we define 
\begin{equation}\label{bisect}
B(X,Y)=\frac{R^j_{ik\ov{\ell}}X^i \ov{X^j} Y^k \ov{Y^\ell}}{\|X\|^2\|Y\|^2}
\end{equation}
 to be the 
\emph{bisectional curvature} of the canonical connection in the directions $X$ and $Y$, which is a real number.
We define the \emph{first Ricci curvature}, the \emph{second Ricci curvature} and the \emph{scalar curvature} of the canonical connection to be the tensors $R_{k\ov{\ell}} = R^i_{ik\ov{\ell}}$, $R'_{k\ov{\ell}}=R^\ell_{ki\ov{i}}$ and $R = R_{k \ov{k}}=R'_{k\ov{k}}$ respectively.
Applying the exterior derivative to the first structure equation, we obtain the \emph{first Bianchi identity},
\begin{equation} \label{eqnfirstBianchi}
d \Theta^i = \Omega^i_j \wedge \theta^j - \theta^i_j \wedge \Theta^j.
\end{equation}
If we define
$M^i_{jk,p}$, $M^i_{jk,\ov{p}}$ by
\begin{equation} \label{derivativeT}
dM^i_{jk} + \theta^i_p M^p_{jk} - M^i_{pk} \theta^p_j - M^i_{jp} \theta^p_k 
 =  M^i_{jk, p} \theta^p + M^i_{jk, \ov{p}} \ov{\theta^p},
\end{equation}
and $N^i_{\ov{j} \, \ov{k}, p}$ and $N^i_{\ov{j} \, \ov{k}, \ov{p}} $ by
\begin{equation}\label{derivativeN}
dN^i_{\ov{j} \, \ov{k}} + \theta^i_p N^p_{\ov{j} \, \ov{k}} - N^i_{\ov{p} \, \ov{k}} \ov{\theta^p_j} - N^i_{\ov{j} \ov{p}} \ov{\theta^p_k}  =  N^i_{\ov{j} \, \ov{k}, p} \theta^p + N^i_{\ov{j} \, \ov{k}, \ov{p}} \ov{\theta^p},
\end{equation}
then the first Bianchi identity implies that (see e.g. \cite{TWY}, section 2.3)
\begin{equation}\label{bianchi}
2M^i_{pj} N^p_{\ov{k} \, \ov{\ell}} + N^i_{\ov{k} \,  \ov{\ell},j}  =  K^i_{j \ov{k} \, \ov{\ell}}.
\end{equation}

We say that the bisectional curvature is bounded
above by $A$ if $$B(X,Y)\leq A$$ holds for all $X,Y\in T'M$.
The first Ricci curvature is
bounded below by a constant $-A_1$ if $$R_{k\ov{\ell}}X^k\ov{X^\ell}\geq -A_1\|X\|^2$$ holds for all $X\in T'M$, and the same
for the second Ricci curvature. The torsion is bounded by $A_2>0$ if 
$$\|\tau(X,Y)\|\leq A_2 \|X\| \|Y\|$$ holds for all $X,Y\in T'M$, and the $(2,0)$ part of the curvature
is bounded by $A_3$ if $$|K^i_{jk\ell}\ov{X^i}Y^j Y^k X^\ell|\leq A_3\|X\|^2\|Y\|^2.$$
\section{The canonical Laplacian}
In this section we study the Laplacian of the canonical connection, and relate it
to the standard Laplacian of the Levi-Civita connection. Again, part of the exposition follows
\cite{TWY}.\\

Let $\nabla$ be the canonical connection of $(M, J, g)$, and $u$ be a function on $M$.  We define the \emph{canonical Laplacian} $\Delta$ of $u$ by
$$\Delta u  = \sum_i \left( (\nabla \nabla u) (e_i, \ov{e_i}) + (\nabla \nabla u) (\ov{e_i}, e_i) \right).$$
This expression is independent of the choice of unitary frame.  Another way to define $\Delta u$ is as follows.  Let $\{ \nu_1, \ldots, \nu_{2n} \}$ be a real local orthonormal frame for $g$ and set
$$\Delta u = \sum_{A=1}^{2n} (\nabla \nabla u)(\nu_A, \nu_A).$$
Clearly this expression is independent of the choice of frame and coincides with the one above.
Now define $u_i$ and $u_{\ov{i}}$ by
\begin{equation} \label{eqndu}
du = u_i \theta^i + u_{\ov{i}} \ov{\theta^i}.
\end{equation}
Writing $\partial u$ and $\ov{\partial} u$ for the $(1,0)$ and $(0,1)$ parts of $du$ respectively we see 
that $\partial u = u_i \theta^i$ and $\ov{\partial} u = u_{\ov{i}} \ov{\theta^i}$.
Define $u_{ik}$, $u_{i\ov{k}}$, $u_{\ov{i} k}$ and $u_{\ov{i} \, \ov{k}}$ by
\begin{align*}
& du_i - u_j \theta_i^j = u_{ik} \theta^k + u_{i \ov{k}} \ov{\theta^k} \\
& du_{\ov{i}} - u_{\ov{j}} \ov{\theta_i^j} = u_{\ov{i} k} \theta^k + u_{\ov{i} \, \ov{k}} \ov{\theta^k}.
\end{align*}
The following lemma is proved in \cite{TWY}.  
\begin{lemma} \label{lemmalap}
\begin{eqnarray} \label{eqnlemma1}
\Delta u  &=&2\sum_i u_{i\ov{i}}\\
& = & - 2 \sum_{i}  (d \partial u)^{(1,1)}(e_i, \ov{e_i}) \\ \label{eqnlemma2}
& = & 2 \sum_{i}  (d \ov{\partial} u)^{(1,1)}(e_i, \ov{e_i}) \\ \label{eqnlemma3}
& = &  \sqrt{-1} \sum_{i}  (d (J d u))^{(1,1)}(e_i, \ov{e_i}),
\end{eqnarray}
where $J$ acts on a 1-form $\alpha$ by $(J\alpha)(X) = \alpha (J (X))$ for a vector $X$.
\end{lemma}

We now want to relate the canonical Laplacian to the standard Levi-Civita Laplacian. In general they
are different, and their precise relation is given by the
\begin{lemma}\label{comparelap}
The Laplacian of the Levi-Civita connection of $g$ acting on a function $u$
is equal to
$$\Delta u + 2M^i_{ji}u_{\ov{j}}+2\ov{M^i_{ji}}u_j.$$
\end{lemma}
\begin{proof}
The Laplacian of the Levi-Civita connection applied to $u$ is given by the trace
of the map $F:T^\mathbb{R} M\to T^\mathbb{R} M$ defined by
$$F(X)=\nabla_X (\grad_g u)+\tau(\grad_g u,X),$$
(see for example \cite{KN} p.282) where $\nabla$ is the canonical connection and $\tau$ is its torsion,
as defined in \eqref{need1}.
To prove the lemma it's enough to show that \eqref{need2} holds. 
We verify this for $X,Y\in T'M$ first. Define functions $X^i_p$, $X^i_{\ov{p}}$,
$X^i_{pq}$, $X^i_{p\ov{q}}$, $X^i_{\ov{p}q}$, $X^i_{\ov{p}\,\ov{q}}$ by
\begin{equation}\label{covariant}
dX^i+ X^j\theta^i_j=X^i_p\theta^p+X^i_{\ov{p}}\ov{\theta^p},
\end{equation}
\begin{equation}\label{covariant3}
dX^i_p+ X^j_p\theta^i_j-X^i_j \theta^j_p=X^i_{pq}\theta^q+X^i_{p\ov{q}}\ov{\theta^q},
\end{equation}
\begin{equation}\label{covariant4}
dX^i_{\ov{p}}+ X^j_{\ov{p}}\theta^i_j-X^i_{\ov{j}} \ov{\theta^j_p}=X^i_{\ov{p}q}\theta^q+X^i_{\ov{p}\,\ov{q}}\ov{\theta^q},
\end{equation}

and similarly for $Y$. Then we have
\begin{equation}\label{covariant2}
\begin{split}
\nabla_X Y&=X^p\nabla_{e_p}(Y^i e_i)=X^p \langle e_p, d Y^i\rangle e_i + X^p Y^i \langle\theta^k_i,e_p\rangle e_k\\
&=-X^p Y^k\langle\theta^i_k,e_p\rangle e_i + X^p Y^i_p e_i+ X^p Y^i \langle\theta^k_i,e_p\rangle e_k=X^p Y^i_p e_i.
\end{split}
\end{equation}
Here and in the following $\langle\cdot,\cdot\rangle$ denotes the pairing between vectors and $1$-forms.
Moreover
\begin{equation*}
\begin{split}
&-\langle \theta^i_\ell,X\rangle\langle\theta^\ell,Y\rangle+\langle \theta^i_\ell,Y\rangle\langle\theta^\ell,X\rangle
+2\Theta^i(X,Y)=2d\theta^i(X,Y)\\
&=X\langle\theta^i,Y\rangle-Y\langle\theta^i,X\rangle-\langle\theta^i,[X,Y]\rangle\\
&=X^j \langle e_j, d Y^i\rangle-Y^k \langle e_k,d X^i\rangle-\langle\theta^i,[X,Y]\rangle\\
&=- X^j Y^k \langle\theta^i_k,e_j\rangle+Y^k X^j\langle\theta^i_j,e_k\rangle+X^j Y^i_j -Y^kX^i_k -\langle\theta^i,[X,Y]\rangle\\
&=-\langle \theta^i_k,X\rangle\langle\theta^k,Y\rangle+\langle \theta^i_j,Y\rangle\langle\theta^j,X\rangle
+\langle\theta^i,\nabla_X Y-\nabla_Y X-[X,Y]\rangle,
\end{split}
\end{equation*}
which shows that the $e_i$ component of $\tau$ is $2\Theta^i$. Similarly
$$2\ov{\Theta^i}(X,Y)=2d\ov{\theta^i}(X,Y)=-\langle\ov{\theta^i},[X,Y]\rangle=\langle\ov{\theta^i},
\nabla_X Y-\nabla_Y X-[X,Y]\rangle,$$
so that the $\ov{e_i}$ component of $\tau$ is $2\ov{\Theta^i}.$

Now we take $X\in T'M$, $Y\in T''M$ (the case when $X,Y\in T''M$ is the same as the one above).
Then 
\begin{equation}\label{need3}
\begin{split}
\nabla_X Y&=X^p\nabla_{e_p}(\ov{Y^i e_i})=X^p \langle e_p, d \ov{Y^i}\rangle \ov{e_i} + X^p \ov{Y^i}\langle
\ov{\theta^k_i},e_p\rangle \ov{e_k}\\
&=-X^p \ov{Y^k}\langle\ov{\theta^i_k},e_p\rangle \ov{e_i} + X^p \ov{Y^i_{\ov{p}} e_i}+ X^p \ov{Y^i}\langle
\ov{\theta^k_i},e_p\rangle \ov{e_k}=X^p \ov{Y^i_{\ov{p}} e_i},
\end{split}
\end{equation}
and similarly $\nabla_Y X=\ov{Y^p} X^i_{\ov{p}} e_i$. Then
\begin{equation*}
\begin{split}
&\langle \theta^i_\ell,Y\rangle\langle\theta^\ell,X\rangle
+2\Theta^i(X,Y)=2d\theta^i(X,Y)=-Y\langle\theta^i,X\rangle-\langle\theta^i,[X,Y]\rangle\\
&=-\ov{Y^k} \langle \ov{e_k},d X^i\rangle-\langle\theta^i,[X,Y]\rangle
=\ov{Y^k} X^j\langle\theta^i_j,\ov{e_k}\rangle-\ov{Y^k}X^i_{\ov{k}} -\langle\theta^i,[X,Y]\rangle\\
&=\langle \theta^i_j,Y\rangle\langle\theta^j,X\rangle
+\langle\theta^i,\nabla_X Y-\nabla_Y X-[X,Y]\rangle,
\end{split}
\end{equation*}
which shows again that the $e_i$ component of $\tau$ is $2\Theta^i$, and the verification of the $\ov{e_i}$ component is analogous.
\end{proof}

A corollary of this is the following observation: if $u$ achieves its infimum at a point $x\in M$, then
$\Delta u(x)\geq 0$. We'll use this remark later.\\

Along the same lines as in Lemma \ref{comparelap}, it's easy to verify that the bisectional curvature
satisfies
$$\frac{1}{2}B(X,Y)\|X\|^2\|Y\|^2= R(V,JV,JW,W),$$
where $R$ is the (real) Riemann curvature tensor of the canonical connection, and
$V=\frac{1}{\sqrt{2}}(X+\ov{X})$, $W=\frac{1}{\sqrt{2}}(Y+\ov{Y})$ are two real tangent
vectors. This quantity is in general different from 
$$R^{LC}(V,JV,JW,W),$$
where $R^{LC}$ is the curvature of the Levi-Civita connection. This is usually
referred to as the holomorphic bisectional curvature \cite{Gr}, but is not very natural
on a general almost-Hermitian manifolds. In fact, it is not hard to see (\cite{Ko}) that 
the bisectional curvature of the canonical connection of an almost-complex submanifold
is always less than the one of the ambient space, but this fails in general for the
Levi-Civita connection (see Proposition 10.1 in \cite{Gr}). The two quantities obviously
agree on a K\"ahler manifold.\\

Let $(M, J, g)$ and $(\ti{M}, \ti{J}, \ti{g})$ be two almost-Hermitian manifolds of dimensions $2n$ and $2\ti{n}$ respectively and let $f : M \rightarrow \ti{M}$ be an \emph{almost-complex mapping},  which means that
$$ \ti{J} \circ f_* = f_* \circ J.$$

We'll also say that $f$ is $(J,\ti{J})$-\emph{holomorphic}.
Then we have the following invariance property.
\begin{lemma}\label{lemmalap2}
For any function $u$ on $\ti{M}$ we have
$$f^* d(\ti{J}du)=d(Jd(u\circ f)).$$
\end{lemma}
\begin{proof}
If $X$ is vector tangent to $M$ then 
\begin{equation*}
\begin{split}
\langle f^*\ti{J}du,X\rangle&=\langle\ti{J}du,f_* X\rangle=\langle du,(\ti{J}\circ f_*)(X)\rangle
=\langle du, f_* JX\rangle\\
&=\langle d(u\circ f),JX\rangle=\langle Jd(u\circ f),X\rangle,
\end{split}
\end{equation*}
and taking the exterior derivative we get the conclusion.
\end{proof}

\section{The Maximum Principle}
In this section we prove a version of Yau's generalized maximum principle \cite{Ya1} for
almost-Hermitian manifolds. The key tool is a Laplacian comparison theorem,
whose analogue in Riemannian and K\"ahler geometry is standard \cite{SY}. 
It was extended to Hermitian manifolds in \cite{CY} and we'll show that it still holds
in our more general setting.\\

The first result is as follows:

\begin{theorem}\label{maxprinc} Let $(M,J,g)$ be a complete almost-Hermitian manifold 
with second Ricci curvature bounded below and with torsion and $(2,0)$ part of the curvature bounded.
Let $u$ be a nonnegative function that is not identically zero and satisfies
\begin{equation}
\Delta u \geq A u^{1+\alpha} - B u,
\end{equation}
where $\alpha, A>0$. Then $u$ is bounded above, $B\geq 0$, and
$$\sup_M u\leq\left(\frac{B}{A}\right)^{\frac{1}{\alpha}}.$$
\end{theorem}
This can be proved exactly in the same way as in \cite{Ya2}, once we have the following:

\begin{prop}[Maximum principle]\label{maxprinc2}
Let $(M,J,g)$ be a complete almost-Hermitian manifold with second Ricci curvature bounded below and 
with torsion and $(2,0)$ part of the curvature bounded.
Let $u$ be a real function that is bounded from below. Then given any $\ve>0$ there exists a point $x_\ve\in M$ such that
$$\liminf_{\ve\to 0}u(x_\ve)=\inf_M u,$$
$$|\nabla u|(x_\ve)\leq\ve,$$
$$\Delta u(x_\ve)\geq-\ve.$$
\end{prop}

The proof of this follows the one in \cite{Ya1} and relies on the

\begin{theorem}[Laplacian comparison]\label{laplcomp}
Let $(M,J,g)$ be a complete almost-Hermitian manifold with second Ricci curvature bounded below
by $-A_1$, torsion bounded by $A_2$ and $(2,0)$ part of the curvature bounded by $A_3$.
Let $\rho$ be the distance from a fixed point $o\in M$. Then at any point where $\rho$ is smooth we have
$$\Delta\rho\leq \frac{2n}{\rho}+C,$$
where $C$ depends only on $A_1,A_2,A_3$ and the dimension of $M$.
Moreover this holds on the whole of $M$ in the sense of distributions.
\end{theorem}

\begin{proof}
Fix a point $x\in M$ outside the cut locus of $o$, and a minimal unit-speed geodesic $\gamma:[0,\rho(x)]\to M$ from $o$ to $x$. 
Let $D\subset \mathbb{C}$ be the unit disc, $z$ be the coordinate on $D$ and $e=\partial/\partial z|_{z=0}$ be the
tangent vector at the origin.
If $v\in T'_x M$ is small enough then Proposition 1.1 in \cite{IR} (see also \cite{NW}) gives a $J$-holomorphic
map $F:D\to M$ with $F(0)=x$ and $F_*(e)=v$, which depends smoothly on $x$ and $v$.
Now extend $v$ to a section $v(t)$ of $T'M$ along $\gamma$, that is small enough and vanishes at $o$. Using Theorem A1 
of \cite{IR} and the compactness of the support of $\gamma$, we can extend $F$ to a smooth family
$F_t:D\to M$ of $J$-holomorphic discs, with the properties that $F_{\rho(x)}=F$, $F_t(0)=\gamma(t)$, $F_{t*}(e)=v(t)$
and $F_0(z)=o$. We'll write $F(t,z)=F_t(z)$ so that we have a map $F:[0,\rho(x)]\times D\to M$. Notice that
we can also allow $v=\gamma'(\rho(x))$.
The vector $F_*(\partial/\partial t)$ belongs to $T^{\mathbb{R}} M\subset T^{\mathbb{C}} M$, and
so it can be written as $T+\ov{T}$ where $T\in T'M$. Moreover the fact that $F_t$ is $J$-holomorphic 
implies that the vector $S=F_*(\partial /\partial z)$ belongs to $T'M$. Notice that
both $T$ and $S$ depend on $(t,z)$ and that $S(t,0)=v(t)$, $(T+\ov{T})(t,0)=\gamma'(t)$.
The map $F$ that we just constructed should be thought of as a $J$-holomorphic variation of $\gamma$, and we are going
to compute the second variation of the arclength. This is the function $L:D\to\mathbb{R}$ defined by
\begin{equation}\label{comput1}
L(z)=\sqrt{2}\int_0^{\rho(x)}\| T(t,z)\| dt,
\end{equation}
which is just the length of the curve $t\mapsto F(t,z)$, that goes from $o$ to $F(\rho(x),z)$, a point near $x$.
Fixing for a moment $(t,z)$, we can take a local unitary frame $\{e_i\}$ near $F(t,z)$ and write $T=T^i e_i$, $S=S^j e_j$.
Then
$$d(\|T\|)=d (T^i \ov{T^i})^{\frac{1}{2}}=\frac{1}{2\|T\|}(T^i_p \ov{T^i}\theta^p+T^i_{\ov{p}}\ov{T^i\theta^p}+T^i \ov{T^i_p\theta^p}+
T^i \ov{T^i_{\ov{p}}}\theta^p),$$
\begin{equation}\label{ddz}
\frac{\del}{\del z}\|T\|=\langle d\|T\|,S\rangle=\frac{T^i_p \ov{T^i}S^p+T^i \ov{T^i_{\ov{p}}}S^p}{2\|T\|}.
\end{equation}
The term $T^i \ov{T^i_{\ov{p}}}S^p$ can be computed as follows: 
$$[T+\ov{T},\ov{S}]=F_*([\partial /\partial t,\partial /\partial \ov{z}])=0,$$ and so
$[T,\ov{S}]=[\ov{S},\ov{T}].$
Combining \eqref{need1}, \eqref{need2}, \eqref{covariant2} and \eqref{need3} we get
$$\theta^i([T,\ov{S}])=-\theta^i(\nabla_{\ov{S}}T)=-\ov{S^p}T^i_{\ov{p}},$$
because $\Theta^i$ has no $(1,1)$ component. But we also have
$$\theta^i([T,\ov{S}])=\theta^i([\ov{S},\ov{T}])=-2\Theta^i(\ov{S},\ov{T})=-2N^i_{\ov{j}\, \ov{k}}\ov{S^j T^k},$$
and so
\begin{equation}\label{useful}
\ov{S^p}T^i_{\ov{p}}=2N^i_{\ov{j}\, \ov{k}}\ov{S^j T^k}.
\end{equation}
By the same token, $\theta^i([T,S])=T^j S^i_j - S^j T^i_j -2M^i_{jk} T^j S^k,$ but we also have that
$\theta^i([T,S])=\theta^i([S,\ov{T}])=-\ov{T^j}S^i_{\ov{j}},$ and so
\begin{equation}\label{useful2}
S^j T^i_j=T^j S^i_j+\ov{T^j}S^i_{\ov{j}}-2M^i_{jk} T^j S^k=\langle\theta^i,\nabla_{T+\ov{T}}S-\tau(T,S)\rangle.
\end{equation}
Also, $[S,\ov{S}]=F_*([\partial /\partial z,\partial /\partial \ov{z}])=0$ implies $\ov{S^p}S^i_{\ov{p}}=0$. 
Using this, we differente \eqref{ddz} once more and we get
\begin{eqnarray} \nonumber\label{comput2}
\frac{\del^2}{\del z\del \ov{z}}\|T\|&=&\left\langle d\left(\frac{T^i_p \ov{T^i}S^p+2\ov{N^i_{\ov{j}\, \ov{k}}}T^i S^j T^k
}{2\|T\|}\right),\ov{S}\right\rangle\\
&=&-\frac{\left|T^i_p \ov{T^i}S^p+2\ov{N^i_{\ov{j}\, \ov{k}}}T^i S^j T^k\right|^2}{4\|T\|^3}
+\frac{T^i_{p\ov{q}} \ov{T^i}S^p\ov{S^q}+T^i_p \ov{T^i_q}S^p\ov{S^q}}{2\|T\|}\\ \nonumber
&&\mbox{} +\frac{2\ov{N^i_{\ov{j}\, \ov{k}, q}}T^i S^j T^k\ov{S^q}+2\ov{N^i_{\ov{j}\, \ov{k}}}T^i_{\ov{q}} S^j T^k\ov{S^q}+
2\ov{N^i_{\ov{j}\, \ov{k}}}T^i S^j T^k_{\ov{q}}\ov{S^q}}{2\|T\|}.
\end{eqnarray}
To deal with the term $T^i_{p\ov{q}} \ov{T^i}S^p\ov{S^q}$ we take the exterior derivative of \eqref{covariant} and
using \eqref{covariant3}, \eqref{covariant4} we get
$$T^j\Omega^i_j=T^i_{pq}\theta^q\wedge\theta^p+T^i_{p\ov{q}}\ov{\theta^q}\wedge\theta^p+
T^i_p\Theta^p
+T^i_{\ov{p}q}\theta^q\wedge\ov{\theta^p}+T^i_{\ov{p}\, \ov{q}}\ov{\theta^q}\wedge\ov{\theta^p}+T^i_{\ov{p}}\ov{\Theta^p},$$
whose $(1,1)$ part gives
$T^i_{p\ov{q}}=T^i_{\ov{q}p}-T^j R^i_{jp\ov{q}},$
and so
\begin{equation}\label{comput3}
T^i_{p\ov{q}}\ov{T^i}S^p\ov{S^q}=T^i_{\ov{q}p}\ov{T^i}S^p\ov{S^q}-T^j\ov{T^i}S^p\ov{S^q} R^i_{jp\ov{q}}.
\end{equation}
The term $T^i_{\ov{q}p}\ov{T^i}S^p\ov{S^q}$ can now be computed as follows:
\begin{equation*}
\begin{split}
0&=\langle d(\ov{S^q} T^i_{\ov{q}} \ov{T^i}-2N^i_{\ov{j}\, \ov{k}}\ov{S^j T^k T^i}), S\rangle=
\ov{S^q} T^i_{\ov{q}p}\ov{T^i} S^p+\ov{S^q} T^i_{\ov{q}} \ov{T^i_{\ov{p}}}S^p\\
&-2N^i_{\ov{j}\, \ov{k},p}\ov{S^j T^k T^i}S^p-2N^i_{\ov{j}\, \ov{k}}\ov{S^j T^k_{\ov{p}} T^i}S^p
-2N^i_{\ov{j}\, \ov{k}}\ov{S^j T^k T^i_{\ov{p}}}S^p,
\end{split}
\end{equation*}
and using \eqref{useful} we get
\begin{equation}\label{comput10}
T^i_{\ov{q}p}\ov{T^i} S^p\ov{S^q}=2N^i_{\ov{j}\, \ov{k},p}\ov{S^j T^k T^i}S^p+2N^i_{\ov{j}\, \ov{k}}\ov{S^j T^k_{\ov{p}} T^i}S^p.
\end{equation}
Combining \eqref{comput2}, \eqref{comput3}, \eqref{comput10},  \eqref{useful2}, \eqref{useful}, \eqref{bianchi}, 
\eqref{switch} and \eqref{bisect} we get
\begin{equation*}
\begin{split}
&\frac{\del^2}{\del z\del \ov{z}}\|T\|\leq\frac{-B(T,S)\|T\|^2\|S\|^2+|\langle\theta^i,\nabla_{T+\ov{T}}S
-\tau(T,S)\rangle|^2+|\langle \theta^i,\tau(\ov{S},\ov{T})\rangle|^2}{2\|T\|}\\
&+\frac{4\textrm{Re}(\ov{K^p_{ikj}}\ov{S^j T^k T^i}S^p-2M^i_{qp} N^q_{\ov{j}\,\ov{k}}\ov{S^j T^k T^i}S^p+
N^i_{\ov{j}\,\ov{k}}\ov{S^j T^i N^k_{\ov{p}\,\ov{q}}}S^p T^q)}{2\|T\|}\\
&\leq \frac{-B(T,S)\|T\|^2\|S\|^2+\|\nabla_{T+\ov{T}}S-\tau(T+\ov{T},S)\|^2+(13A_2^2+4A_3) \|S\|^2\|T\|^2}{2\|T\|}.
\end{split}
\end{equation*}
All the terms on the right hand side are tensorial, and hence independent of the choice of unitary frame.
Combining this with \eqref{comput1} and setting $z=0$ we finally get
\begin{equation}\label{comput8}
\begin{split}
\frac{\del^2}{\del z\del \ov{z}}L(z)\Big|_{z=0}&\leq\frac{1}{2}\int_0^{\rho(x)}
\|\nabla_{\gamma'(t)} v(t)+\tau(v(t),\gamma'(t))\|^2 dt\\
&+\frac{1}{2}\int_0^{\rho(x)}(C- B(\gamma'(t),v(t)))\|v(t)\|^2 dt,
\end{split}
\end{equation}
where $C=13A_2^2+4A_3$. Notice that the right hand side is homogeneous
of degree 2 in $v(t)$, so it doesn't matter that we had picked $v(t)$ small in the first place. Define 
$G:D\to M$ to be the $J$-holomorphic disc $G(z)=F(\rho(x),z)$, originally called $F$, and notice that
since $\gamma$ is minimizing we have
$L(z)\geq (\rho\circ G)(z)$ and $L(0)=(\rho\circ G)(0)$, hence
\begin{equation}\label{comput6}
\frac{\del^2}{\del z\del \ov{z}}L(z)\Big|_{z=0}\geq \frac{\del^2}{\del z\del \ov{z}}(\rho\circ G)(z)\Big|_{z=0}.
\end{equation}
But now Lemma \ref{lemmalap} and Lemma \ref{lemmalap2} imply that
\begin{equation}\label{comput7}
\begin{split}
\sqrt{-1}d(Jd\rho)^{(1,1)}(v,\ov{v})&=\sqrt{-1}d(J_{\mathbb{C}} d(\rho\circ G))(e,\ov{e})\\
&=4\frac{\del^2}{\del z\del \ov{z}}(\rho\circ G)(z)\Big|_{z=0},
\end{split}
\end{equation}
where $J_{\mathbb{C}}$ is the standard complex structure on $\mathbb{C}$. 

Now we pick $v$ to have unit length, and we choose $v(t)$ to be of the form $v(t)=f(t)w(t)$ where $w(t)$ is the parallel transport with respect to $\nabla$ 
of $v$ along $\gamma$, and $f(t)\geq 0$ is a smooth increasing function that satisfies $f(0)=0$ and $f(\rho(x))=1$. 
 Then, using \eqref{comput8} and the fact that $\nabla g=0$, we get
\begin{equation}\label{comput9}
\begin{split}
\frac{\del^2}{\del z\del \ov{z}}L(z)\Big|_{z=0}&\leq \frac{1}{2}\int_0^{\rho(x)}
\|f'(t)w(t)+f(t)\tau(w(t),\gamma'(t))\|^2dt\\
&+\frac{1}{2}\int_0^{\rho(x)}f(t)^2(C- B(\gamma'(t),w(t)))dt\\
&\leq\frac{1}{2}\int_0^{\rho(x)}\left(f'(t)^2+2A_2 f(t)f'(t)+A_2^2 f(t)^2\right)dt\\
&+\frac{1}{2}\int_0^{\rho(x)}f(t)^2(C- B(\gamma'(t),w(t)))dt.
\end{split}
\end{equation}
Now we combine \eqref{comput6}, \eqref{comput7}, \eqref{comput9} and sum them up when $v$ ranges 
in $v_1,\dots,v_n$, a unitary basis of $T_x'M$, and using Lemma \ref{lemmalap} we get
$$\Delta\rho(x)\leq 2\int_0^{\rho(x)}\left(nf'(t)^2+2nA_2 f(t)f'(t)+C'f(t)^2\right)dt,$$
where $C'=nC+A_1+nA_2^2$.
Next, following \cite{CY}, we pick $$f(t)=\left(\frac{t}{\rho(x)}\right)^\alpha,$$
where $\alpha>1$ will be chosen presently. With this choice we can easily compute that
$$\Delta\rho(x)\leq \frac{2n}{\rho(x)}+2nA_2+\frac{2n(\alpha-1)^2}{(2\alpha-1)\rho(x)}+\frac{2C'}{2\alpha+1}\rho(x).$$
Now we choose $\alpha$, depending on $\rho(x)$, such that the last two terms on the
right hand side are equal. Thus
$$\frac{n(\alpha-1)^2}{(2\alpha-1)\rho(x)}+\frac{C'}{2\alpha+1}\rho(x)=2\sqrt{\frac{(\alpha-1)^2}{4\alpha^2-1}nC'}\leq\sqrt{nC'},$$
which is what we want. Finally to show that the inequality in the sense of distributions holds on the whole manifold
we can just follow the argument in \cite{SY}, pag.7.
\end{proof}

\begin{proof}[Proof of Proposition \ref{maxprinc2}.] Once we have established the Laplacian comparison
Theorem \ref{laplcomp}, the proof is standard, but we include it for completeness.
We'll use a trick due to Calabi \cite{Ca} to avoid the cut locus of $o$. If the infimum of $u$ is attained in the geodesic ball of radius $1$ centered at $o$ then there's nothing
to prove, so that we may assume that $\rho>1$. Then Theorem \ref{laplcomp} gives $\Delta\rho\leq C$ for
a uniform constant $C$.
For any $\ve>0$ the function $u+\ve\rho$ attains its infimum at a point $x_\ve\in M$. Let $\gamma$ be a minimal
unit-speed geodesic from $o$ to $x_\ve$, $\ti{x}$ be another point on $\gamma$ and
denote by $\ti{\rho}$ the distance from $\ti{x}$. Then for any $x\in M$ we have
$$u(x)+\ve\ti{\rho}(x)=u(x)+\ve\rho(x)-\ve\rho(x)+\ve\ti{\rho}(x)\geq u(x)+\ve\rho(x)-\ve\rho(\ti{x}),$$
and taking the infimum over $x$ we get
$$\inf_M (u+\ve \ti{\rho})\geq u(x_\ve)+\ve\rho(x_\ve)-\ve\rho(\ti{x})=u(x_\ve)+\ve\ti{\rho}(x_\ve).$$
Hence the function $u+\ve \ti{\rho}$ also attains its infimum at $x_\ve$. But we can now choose $\ti{x}$ outside the cut locus of $x_\ve$,
so that $\ti{\rho}$ is smooth at $x_\ve$, and using the remark after Lemma \ref{comparelap} we get
$$|\nabla u|(x_\ve)=\ve |\nabla \ti{\rho}|(x_\ve)=\ve,$$
$$\Delta u(x_\ve)\geq -\ve\Delta \ti{\rho}\geq -\ve C.$$
Finally we check that $\liminf_{\ve\to 0}u(x_\ve)=\inf_M u.$ If not, there exist $\ov{x}\in M$ and $\delta>0$ such that
$u(\ov{x})<u(x_\ve)-\delta$ for all $\ve$ small. We still have
$$u(\ov{x})+\ve\rho(\ov{x})\geq u(x_\ve)+\ve\rho(x_\ve).$$ If $\rho(x_\ve)$ is bounded then we can take a convergent
subsequence of points and letting $\ve\to 0$ we get a contradiction. If $\rho(x_\ve)$ is unbounded, we take $\ve$ small
so that $\rho(x_\ve)>\rho(\ov{x})$ and get
$$u(x_\ve)+\ve\rho(x_\ve)\leq u(\ov{x})+\ve\rho(\ov{x})<u(x_\ve)-\delta+\ve\rho(x_\ve),$$
which again is absurd.\end{proof}

\begin{proof}[Proof of Theorem \ref{maxprinc}]
Now that we have Proposition \ref{maxprinc2}, the argument is exactly the same as in \cite{Ya2} so we'll just 
sketch it. One defines a function $$v=(u+c)^{-\frac{\alpha}{2}},$$
where $c>0$ is fixed. 
Since $v$ is bounded below we can apply Proposition \ref{maxprinc2}
and for any $\ve>0$ we get a point $x_\ve\in M$ where we have
$$Au^{1+\alpha}-Bu\leq \Delta u\leq\frac{2}{\alpha}\left((u+c)^{\frac{\alpha+2}{2}}+\ve\frac{\alpha+2}{\alpha}
(u+c)^{1+\alpha}\right)\ve.$$
If $\sup_M u=+\infty$ then we can let $\ve\to 0$ in the last inequality and get a contradiction.
So $\sup_M u<+\infty$ and again letting $\ve\to 0$ we get the conclusion.
\end{proof}

\begin{remark} Instead of our Theorem \ref{laplcomp} we could have used the standard Laplacian comparison,
as in \cite{Ya1}. This gives a similar result for the Laplacian of the Levi-Civita connection, under the
assumption that the Ricci curvature of the Levi-Civita connection is bounded below. Notice that to apply this
to our situation we still need the assumption that the torsion be bounded, to compare the two Laplacians as
in Lemma \ref{comparelap}. The reason why we chose not to do this is because in our main theorems we don't want 
any assumption on the Levi-Civita connection, but only on the canonical connection.
\end{remark}

\section{The Schwarz Lemma}
In this section we prove Theorems \ref{main1} and \ref{main2}. Using Cartan's formalism of moving
frames, and the canonical connection, we prove in \eqref{final} a generalization of a formula due to Chern and Lu
\cite{Lu} in the integrable case. The Schwarz Lemma then follows at once from the maximum principle, Theorem \ref{maxprinc}.
The corresponding formula for the volume form is much easier, and already appears in \cite{GH2}.\\

Let $(M, J, g)$ and $(\ti{M}, \ti{J}, \ti{g})$ be two almost-Hermitian manifolds of dimensions $2n$ and $2\ti{n}$ respectively and let $f : M \rightarrow \ti{M}$ be an almost-complex mapping.
Let $\{ e_i \}$ and $\{ \theta^i \}$ be local unitary frames and coframes for $g$ 
on $M$ and let $\{ \ti{e}_i \}$ and $\{ \ti{\theta}^i \}$ be those for $\ti{g}$ on $\ti{M}$. 
 Let $\nabla$ and $\tilde{\nabla}$ be the canonical connections for $(M,J,g)$ and $(\ti{M}, \ti{J}, \ti{g})$ 
respectively.  We will use $\theta^i_j$, $\Theta^i$, $\Omega^i_j$ and 
$\ti{\theta}^\alpha_\beta$, $\ti{\Theta}^\alpha$, $\ti{\Omega}^\alpha_\beta$ to denote the connection 1-forms, 
torsion and curvature for $\nabla$ and $\tilde{\nabla}$ respectively.  
Here we use roman letters $i,j,k, \ldots = 1, 2, \ldots, n$ for indices 
on $M$ and greek letters $\alpha, \beta, \ldots = 1,2, \ldots, \ti{n}$ for indices on $\ti{M}$.

Since $f$ is almost-complex, there exist functions $a^{\al}_i$ on $M$ such that
\begin{equation} \label{eqna}
f^*\ti{\theta}^{\al} = a^{\al}_i \theta^i.
\end{equation}
Define a function $u$ by $u = \textrm{tr}_g (f^*{\ti{g}})$.  Locally we can write $u$ as
$$ u = a_i^{\alpha} \ov{a_i^{\alpha}}.$$
From now on, we will often omit writing the pullback $f^*$.  
Differentiating (\ref{eqna}) and using the first structure equations for $\nabla$ and $\ti{\nabla}$ we obtain
\begin{eqnarray} \nonumber
d \ti{\theta}^{\al} & = & da_i^{\al} \wedge \theta^i -  a_i^{\al} \theta_j^i \wedge \theta^j + a_i^{\al} \Theta^i \\
& = & - a_i^{\be} \ti{\theta}_{\beta}^{\al} \wedge \theta^i + \ti{\Theta}^{\al}.
\end{eqnarray}
Rearranging this gives
\begin{equation}
(da_i^{\al} +  \ti{\theta}_{\be}^{\al} a_i^{\be} - a_j^{\al} \theta_i^j ) \wedge \theta^i =
\ti{\Theta}^{\al} - a_i^{\al} \Theta^i .
\end{equation}
Since the right hand side has no $(1,1)$ component, it follows that we can define functions $a_{ik}^{\al}$ by
\begin{equation} \label{eqna2}
d a_i^{\al} + \ti{\theta}_{\be}^{\al} a_{i}^{\be} - a_j^{\al} \theta_i^j = a^{\al}_{ik} \theta^k.
\end{equation}
Now apply the exterior derivative to both sides of this equation, substitute from the structure equations and (\ref{eqna2}), and cancel some terms to obtain
\begin{eqnarray*}
\lefteqn{a_i^{\beta} \ti{\Omega}_{\beta}^{\al} + a_{ik}^{\beta} \theta^k \wedge \ti{\theta}_{\be}^{\al} - a_j^{\al} \Omega_i^j - a_{jk}^{\al} \theta^k \wedge \theta_i^j} \\
& = & da_{ik}^{\al} \wedge \theta^k + a_{ik}^{\al} (- \theta_j^k \wedge \theta^j + \Theta^k),
\end{eqnarray*}
which can be rewritten as
\begin{eqnarray} \nonumber
\lefteqn{ ( d a_{ik}^{\al} - a_{ij}^{\al} \theta_k^j + a_{ik}^{\beta} \ti{\theta}_{\beta}^{\al} - a_{jk}^{\al} \theta_i^j ) \wedge \theta^k } \\ \label{eqnda2}
& = & a_i^{\beta} \ti{\Omega}_{\beta}^{\al} - a_j^{\al} \Omega_i^j - a_{ik}^{\al} \Theta^k.
\end{eqnarray}
Define functions $a_{ik\ell}^{\al}$ and $a^{\al}_{ik\ov{\ell}}$ by
\begin{equation} \label{eqnda3}
da^{\alpha}_{ik} - a^{\al}_{ij} \theta_k^j + a^{\be}_{ik} \ti{\theta}^{\al}_{\be} - a^{\al}_{jk} \theta_i^j = a^{\al}_{ik\ell} \theta^\ell + a^{\al}_{ik\ov{\ell}} \ov{\theta^\ell}.
\end{equation}
Then taking the $(1,1)$ part of (\ref{eqnda2}) we obtain
\begin{equation} \label{eqncurv1}
a^{\al}_{ik\ov{\ell}} \theta^k \wedge \ov{\theta^\ell} = -a_i^{\be} \ti{R}_{\beta \ga \ov{\de}}^{\al} \ti{\theta}^{\ga} \wedge \ov{\ti{\theta}^{\de}} + a_j^{\al} R_{ik\ov{\ell}}^j \theta^k \wedge \ov{\theta^\ell}.
\end{equation}
We now wish to calculate $du$.  Using (\ref{eqna2}) we have
$$du =  \ov{a_i^{\al}} a^{\al}_{ik} \, \theta^k + a_i^{\al} \ov{a_{ik}^{\al}\theta^k},$$
which means  $\partial u =  \ov{a_i^{\al}} a^{\al}_{ik} \, \theta^k$,
$\ov{\partial}u = a_i^{\al} \ov{a_{ik}^{\al}\theta^k}.$
Then
\begin{eqnarray*}
d \partial u & = & a_{ik}^{\al} d \ov{ a_i^{\al}} \wedge \theta^k + 
\ov{a_i^{\al}} da_{ik}^{\al} \wedge \theta^k + \ov{a_i^{\al}} a^{\al}_{ik} d\theta^k \\
& = & a_{ik}^{\al} \ov{a_{i\ell}^{\al}\theta^\ell} \wedge \theta^k + 
\ov{a_i^{\al}} ( a^{\al}_{ik\ell} \theta^\ell + a^{\al}_{ik \ov{\ell}} \ov{\theta^\ell} )\wedge \theta^k + \ov{a_i^{\al}} a_{ik}^{\al} \Theta^k ,
\end{eqnarray*}
where we have used (\ref{eqna2}), (\ref{eqnda3}) and the first structure equation.  Hence
\begin{equation}
(d\partial u)^{(1,1)} = - a_{ik}^{\al} \ov{a_{i\ell}^{\al}} \theta^k \wedge \ov{\theta^\ell} - \ov{a_i^{\al}} a^{\al}_{ik\ov{\ell}} \theta^k \wedge \ov{\theta^\ell}.
\end{equation}
Substituting from (\ref{eqncurv1}) we have
\begin{eqnarray} \nonumber
(d\partial u)^{(1,1)} & = & \left(- a_{ik}^{\al} \ov{a_{i\ell}^{\al}}  + \ov{a_i^{\al}} a_i^{\be}
 \ti{R}^{\al}_{\be \ga \ov{\de}} a^{\ga}_k \ov{a^{\de}_\ell}
 - \ov{a_i^{\al}} a_j^{\al} R_{ik\ov{\ell}}^j\right) \theta^k \wedge \ov{\theta^\ell}.
\end{eqnarray}
Then from Lemma \ref{lemmalap} we obtain
\begin{eqnarray}\label{final}
\frac{1}{2} \Delta u & = &  |a^{\alpha}_{ik}|^2 - \ov{a_i^{\al}} a_i^{\be} a_k^{\ga} \ov{a_k^{\de}} \ti{R}^{\al}_{\be \ga \ov{\de}} +
 \ov{a_i^{\al}} a_j^{\al} R'_{i \ov{j}} .
\end{eqnarray}
If the second Ricci curvature of $g$ is bounded below by $-K_1$ and the bisectional curvature of $\ti{g}$ is bounded
above by $-K_2<0$, then we get
$$\frac{1}{2} \Delta u\geq K_2 u^2-K_1 u.$$
Then Theorem \ref{maxprinc} gives that
$$\tr{g}{f^*\ti{g}}=u\leq \frac{K_1}{K_2},$$
which proves Theorem \ref{main1} since $f^*\ti{g}\leq u g$.\\

Now assume that $M$ and $\ti{M}$ have the same dimension $2n$. 
Define a function 
$$v=\frac{\det f^*\ti{g}}{\det g},$$
so that $f^* dV_{\ti{g}}=v dV_g.$ Then $f$ is non-degenerate precisely when $v>0$ and is
totally degenerate when $v\equiv 0.$ Locally $v=|\nu|^2$ where
$\nu=\det(a^\alpha_i).$
A computation in section 3 of \cite{GH2} (see also Lemma 3.2 in \cite{TWY}) gives
$$\frac{1}{2}\Delta v=vR-v\ti{R}_{\alpha\ov{\beta}}a^\alpha_i\ov{a^\beta_i}.$$
So if the scalar curvature of $g$ is bounded below by $-nK_1$ and the first Ricci curvature of $\ti{g}$ is bounded
above by $-K_2$, with $K_2>0$, then we get
$$\frac{1}{2}\Delta v\geq K_2 uv-nK_1 v\geq nK_2 v^{1+\frac{1}{n}}-nK_1 v,$$
where we used the arithmetic-geometric mean inequality.
Then Theorem \ref{maxprinc} gives that
$$v\leq \left(\frac{K_1}{K_2}\right)^n,$$
which proves Theorem \ref{main2}.

\section{Product of almost-complex manifolds}
In this section we prove Theorem \ref{main3}. We adapt the argument in \cite{SZ} to our case,
using again local holomorphic discs instead of complex coordinates, and applying our 
Theorem \ref{main1} and Proposition \ref{maxprinc2}.\\

Suppose $M=X\times Y$ is the product of two almost-complex manifolds of (real) dimensions $2n$ and $2m$ respectively.
Assume for a contradiction that $M$ admits a complete almost-Hermitian metric $g$ with torsion and $(2,0)$ part of the 
curvature bounded, and with bisectional curvature bounded between two negative constants, so that
$$-C_1<B(V,W)<-C_2<0$$
holds for all $V,W\in T'M$. Fix a point $q\in Y$ and pick 
$F:D\to Y$ a $J$-holomorphic disc with $F(0)=q$ and $F_*(e)\neq 0$.
Here again $D\subset \mathbb{C}$ is the unit disc and the existence of such a map is given by \cite{IR}.
Moreover, up to shrinking the disc, we can assume that the $F$ is an immersion, so that the 
vector field $V=F_*(\partial/\partial z)\in T'Y$ doesn't vanish on the image of $F$,
and that $T'Y$ can be trivialized in a neighborhood of the image.
For each $x\in X$ define a map
$G_x:D\to M$ by sending $z$ to $(x,F(z))$. Each $G_x$ is almost-complex with respect to the
given almost-complex structures and moreover the map $G:X\times D\to M$ given by $G(x,z)=G_x(z)$ is also
almost-complex. Take $\eta\in C^\infty_c(D)$ to be a smooth nontrivial cutoff function,
with $0\leq \eta\leq 1$, and define a smooth positive function $f$ on $M$ by
$$f(x,y)=f(x)=\int_D \eta G_x^*g.$$
Equip $D$ with the Poincar\'e metric $g_0$,
and apply Theorem \ref{main1} to $G_x$ to get
$$G_x^*g\leq \frac{1}{C_2}g_0,$$
which implies that $f$ is bounded above. Now fix a point $p=(x_0,q)\in M$ and pick $\{e_1,\dots,e_n\}$ a local
frame on $X$ around $x_0$, and $\{e_{n+1},\dots,e_{n+m}\}$ a local frame on $Y$ around the image of $F$. Then, by abusing
notation, we denote by $\{e_1,\dots,e_{n+m}\}$ the induced frame on $M$, which in general is not unitary.
Then locally the Hermitian metric $g$ on $T'M$ can be written as $g_{i\ov{j}}\theta^i\otimes\ov{\theta^j}$,
and on the image of $G_x$, $V$ is of the form $V=V^je_j$, where $n+1\leq j\leq n+m$. Moreover we can assume that at $p$ we have
$g_{i\ov{j}}=\delta_{ij}$ for $1\leq i,j\leq n$.
Then we can write
$$f=\sqrt{-1}\int_D\eta g_{j\ov{k}}V^j\ov{V^k}dz\wedge d\ov{z}.$$
Since $f$ is constant along $Y$,
we see that $f_j=f_{\ov{j}}=0$ for $n+1\leq j\leq n+m$. From now on fix $1\leq i\leq n$, and notice
that on $X\times D$ we have, by abusing notation,
$[\ov{e_i},\partial/\partial z]=0$. Hence $$0=G_*([\ov{e_i},\partial/\partial z])=[\ov{e_i},V],$$
and so \eqref{need1} gives $0=\theta^j([\ov{e_i},V])=V^j_{\ov{i}}$ for all $n+1\leq j\leq n+m$.
Hence
$$f_i=\sqrt{-1}\int_D \eta g_{j\ov{k}}V^j_i \ov{V^k}dz\wedge d\ov{z},$$
\begin{equation}\label{compt}
f_{i\ov{i}}=\sqrt{-1}\int_D \eta g_{j\ov{k}}\left(V^j_{i\ov{i}}\ov{V^k}+V^j_i\ov{V^k_i}\right)dz\wedge d\ov{z}
\geq \sqrt{-1}\int_D \eta g_{j\ov{k}}V^j_{i\ov{i}}\ov{V^k}dz\wedge d\ov{z},
\end{equation}
where we have used that $\nabla g=0$.
Now proceeding as in the derivation of \eqref{comput3}, we get
\begin{equation}\label{compt2}
\begin{split}
g_{j\ov{k}}V^j_{i\ov{i}}\ov{V^k}&=g_{j\ov{k}}\left(V^j_{\ov{i}i}\ov{V^k}-V^\ell\ov{V^k} R^j_{\ell i\ov{i}}\right)=
-g_{j\ov{k}}V^\ell\ov{V^k} R^j_{\ell i\ov{i}}\\
&\geq C_2 g_{j\ov{k}}V^j\ov{V^k}g_{i\ov{i}}.
\end{split}
\end{equation}
Denote by $h$ the almost-Hermitian metric on $X$ obtained by restricting $g$ to $X\times\{q\}$.
In \cite{Ko} it is proved that the bisectional curvature of an almost-complex submanifold is
always less than the one of the ambient space, and so the bisectional curvature of $h$
is bounded above by $-C_2$. The projection $\pi_1:(M,g)\to (X,h)$ is almost-complex and
Theorem \ref{main1} gives
$$\pi_1^*h\leq C_3 g,$$
where $C_3=\frac{(n+m)C_1}{C_2}$. This implies that
\begin{equation}\label{compt3}
g_{i\ov{i}}(x,q)\leq C_3 g_{i\ov{i}}(x,y)
\end{equation}
for any $(x,y)$ near $p$. Combining \eqref{compt}, \eqref{compt2} and \eqref{compt3} we get
$$f_{i\ov{i}}(x,y)\geq \frac{C_2}{C_3}g_{i\ov{i}}(x,q) f(x,y),$$
and so at $p$ we get $f_{i\ov{i}}\geq \alpha f$, where $\alpha=\frac{C_2}{C_3}>0$. Summing up 
and using Lemma \ref{lemmalap}
we get $\frac{1}{2}\Delta f\geq n\alpha f,$ and Proposition \ref{maxprinc2} applied to $-f$ gives $f=0$, which is absurd.

\end{document}